\documentclass{article}
\usepackage{mathrsfs}
\usepackage{color}
\usepackage{epsfig, graphicx}
\usepackage{latexsym,amsfonts,amsbsy,amssymb}
\usepackage{amsmath,amsthm}
\usepackage{framed}
\usepackage{setspace}
\makeatletter 

\@addtoreset{equation}{section}
\makeatother 
\newtheorem{theorem}{Theorem}[section]
\newtheorem{lemma}{Lemma}[section]
\newtheorem{corollary}{Corollary}[section]

\newtheorem{definition}{Definition}[section]

\def\E{{\mathcal{E}}}
\def\T{{\mathcal{T}}}

\def\dQ{{\mathbb{Q}}}

\def\b0{\boldsymbol{0}}
\def\sumT{\sum_{T\in\mathcal{T}_h}}     

\def\bn{{\mathbf{n}}}

\def\bf{{\mathbf{f}}}
\def\bq{{\mathbf{q}}}

\newtheorem{algorithm1}{Weak Galerkin Algorithm}

 \newcommand{\eps}{\varepsilon}

 \newcommand{\Real}{\mathbb{R}}

 \newcommand{\trb}[1]{|\!|\!|#1|\!|\!|}

\allowdisplaybreaks 

\begin{document}
\title{Acceleration of weak Galerkin methods for the Laplacian eigenvalue problem}
\date{}
\author{
Qilong Zhai\thanks{Department of Mathematics, Jilin University, Changchun, 130012,
China (diql15@mails.jlu.edu.cn)},\ \ 
Hehu Xie\thanks{LSEC and Institute of Computational Mathematics and Scientific/Engineering Computing,
Academy of Mathematics and Systems Science, Chinese Academy of
Sciences, Beijing 100190, China (hhxie@lsec.cc.ac.cn). The research
of this author was supported in part by the National Natural Science
Foundation of China (NSFC) under grants 91330202, 11001259,
11371026, 11201501, 11031006, 2011CB309703, Science Challenge Project (No. JCKY2016212A502),
and the National Center
for Mathematics and Interdisciplinary Science, CAS},\ \ 
Ran Zhang\thanks{Department of Mathematics, Jilin University, Changchun, 130012, China
(zhangran@mail.jlu.edu.cn). The research of this author was supported
in part by the   National Natural Science Foundation of China (NSFC 1271157, 11371171, 11471141),
and by the New Century Excellent Researcher Award Program from Ministry of Education of
China}\ \ \  and \ \ 
Zhimin Zhang\thanks{Beijing Computational Science Research Center, Beijing,
100193, China (zmzhang@csrc.ac.cn); Department of Mathematics, Wayne State University,
Detroit, MI 48202 (zzhang@math.wayne.edu).
The research of this author was supported in part by the National
Natural Science Foundation of China (NSFC 11471031, 91430216)
and the U.S. National Science Foundation (DMS--1419040)}}
\maketitle
\begin{abstract}
Recently, we proposed a weak Galerkin finite element method for the Laplace
eigenvalue problem. In this paper, we present two-grid and
two-space skills to accelerate the weak Galerkin method.
By choosing parameters properly, the two-grid and two-space weak Galerkin method
not only doubles the convergence rate, but also maintains the asymptotic lower
bounds property of the weak Galerkin method. Some numerical examples
are provided to validate our theoretical analysis.
\end{abstract}



\section{Introduction}

The eigenvalue problem arises from many branches of mathematics and
physics, including quantum mechanics, fluid mechanics, stochastic process
and structural mechanics. A variety of applications of eigenvalue problem, especially
the Laplacian eigenvalue problem, are surveyed by a recent SIAM review paper \cite{Grebenkov2013}.
Many numerical methods have been developed for the
Laplacian eigenvalue problem, such as finite difference methods
\cite{Weinberger1958} and finite element methods
\cite{I.Babuska,Boffi2010,Chatelin2011}.

The finite element method is one of the efficient approaches for the
Laplacian eigenvalue problem for its simplicity and adaptivity on
triangular meshes. Due to the minimum-maximum principle, the
conforming finite element method always gives the upper bounds for the Laplacian
eigenvalues. In order to get accurate intervals for eigenvalues, it
is necessary to have lower bounds of  eigenvalues. There are
mainly two ways, the post-processing method
\cite{Carstensen,Larson2000,Liu2015,Liu2013a,Liu2015,Tomas2,Tomas1} and the nonconforming finite element  method
\cite{AD04,Hu2014b,Luo2012,E2015}. Some specific nonconforming finite element methods provide asymptotic lower bounds
for eigenvalues without solving an auxiliary problem, while
it seems difficult to construct a high order nonconforming element.

{\color{black}
Recently, a new method for solving the partial differential equations,
named the weak Galerkin (WG) method, has been developed. The WG
method was first introduced in \cite{Wang2014a} for the
second order elliptic equation, and was soon applied to many types of partial
differential equations, such as the parabolic equation \cite{Li2013},
the biharmonic equation
\cite{Oden2007,Mu2013,Zhang2015},
the Brinkman equation \cite{Mu}, and the Maxwell equation \cite{Mu2013d}.
In \cite{Xie2015}, the Laplacian eigenvalue problem was
investigated by the WG method.} An astonishing feature is:
it offers asymptotic lower bounds for the
Laplacian eigenvalues on polygonal meshes by employing high order polynomial elements.
Comparing with the boundary value problems, the eigenvalue problem
is more difficult to solve since it is actually a special nonlinear
equation. Solving eigenvalue problems need more computational work
and memory than solving corresponding boundary value problems. So how to
accelerate the solving speed is a necessary and important
topic in computational mathematics.

Two-grid and two-space methods are both efficient numerical methods for nonlinear
problems. The main idea is to approximate a
large nonlinear system by solving a small nonlinear system and a
large linear system, and thereby to reduce the computational cost. The
two-grid method was first introduced in \cite{Xu1994} to solve a semilinear second order elliptic
problem. Soon it was adopted for different kinds of PDEs \cite{Marion1995,Xu1996,Xu2000}. The eigenvalue problem can also be
viewed as a nonlinear problem, the corresponding two-grid method
was studied in \cite{Xu2001}, {\color{black}and some variations have been developed later,
such as the shifted-inverse power method \cite{Hu2015,Yang2011}, some applications have also been developed
for Stokes
\cite{Chen2009,Xie2015a,Feng2014}
 and Maxwell eigenvalue problems \cite{MartinJ.WilcoxL.C.BursteddeC.Ghattas2012},
the second order elliptic eigenvalue problems by the mixed finite element methods \cite{Chen2011},
 Bose-Einstein problems \cite{Osher2009}.
The two-space method is proposed 
for the biharmoinc eigenvalue problem
by the nonconforming finite element methods. Then it is adopted for the Laplacian eigenvalue problems
by the conforming finite element methods \cite{Racheva2002} and Stokes eigenvalue problems \cite{Chen2009}.}

In this paper, we apply the two-grid \cite{Xu2001} and two-space 
methods to
accelerate the WG method for the Laplacian eigenvalue
problems.  In this way, the computing complexity of the WG method can be reduced
greatly. Another important nice feature is: by choosing the mesh sizes
properly, the two-grid WG method can still provide lower
bounds for the Laplacian eigenvalues. Rigorous theoretical analysis will be given for
the proposed method, and numerical examples will be provided as well.

An outline of the paper goes as follows. In Section 2, we introduce
the WG method for the eigenvalue problem and the
corresponding basic error estimates.
In Section 3, we give the $H^{-1}$ error estimate for the WG method, which
plays an important role in the analysis.
The two-grid method will be
introduced and analyzed in Section 4. Section 5 is devoted to the
two-space method.
 In Section 6, some numerical examples are presented to validate our theoretical analysis.
 Some concluding remarks are given in the final section.

\section{A standard discretization of weak Galerkin scheme}
In this section, we state some notation in this paper,
introduce the standard WG scheme for Laplacian eigenvalue problem briefly
and present some results from \cite{Xie2015}.
Throughout this  paper, we always use $C$ to represent a constant independent
of mesh sizes $H$ and $h$, which may have different values according to the occurrence.
The symbol $a\lesssim b$ stands for $a\le C b$ for some constant $C$.

In this paper, for simplicity, we consider the following Laplacian eigenvalue problem:
Find $(\lambda, u)$ such that
\begin{equation}\label{problem-eq}
\left\{
\begin{array}{rcl}
-\Delta u &=& \lambda u,\quad \text{in }\Omega,\\
u &=& 0,\quad~~ \text{on }\partial\Omega,\\
\int_\Omega u^2 &=& 1,
\end{array}
\right.
\end{equation}
where $\Omega$ is a polygon region in $\Real^d$ $(d=2,3)$.

The standard Sobolev space notation are also used in this paper.
 Let $D$ be any open bounded
domain with Lipschitz continuous boundary in $\mathbb{R}^d$ $(d=2, 3)$.
We use the standard definition for the Sobolev space $H^s(D)$ and
their associated inner products $(\cdot, \cdot)_{s, D}$, norms
$\|\cdot\|_{s, D}$, and seminorms $|\cdot|_{s, D}$ for any $s\ge 0$.
For example, for any integer $s\ge 0$, the seminorm $|\cdot|_{s, D}$
is given by
$$
|v|_{s, D}=\left(\sum_{|\alpha|=s}\int_D |\partial^\alpha v|^2 {\rm
d} D\right)^{\frac{1}{2}}
$$
with the usual multi-index notation
$$
\alpha=(\alpha_1, \cdots, \alpha_d), \quad
|\alpha|=\alpha_1+\cdots+\alpha_d, \quad
\partial^\alpha=\prod_{j=1}^d \partial_{x_j}^{\alpha_j}.
$$
The Sobolev norm $\|\cdot\|_{m,D}$ is given by
$$
\|v\|_{m, D}=\left(\sum_{j=0}^{m} |v|_{j, D}^2\right)^{\frac{1}{2}}.
$$

The space $H^0(D)$ coincides with $L^2(D)$, for which the norm and
the inner product are denoted by $\|\cdot\|_D$ and
$(\cdot,\cdot)_D$, respectively. When $D=\Omega$, we shall drop the
subscript $D$ in the norm and in the inner product notation.


Let $\mathcal T_h$ be a partition of the domain $\Omega$, and the elements
in $\T_h$ are polygons satisfying the regular assumptions specified
in \cite{Wang2014a}. Denote by $\E_h$ the edges in $\T_h$, and by
$\E_h^0$ the interior edges $\E_h\backslash \partial\Omega$. For
each element $T\in\T_h$, $h_T$ represents the diameter of $T$, and
$h=\max_{T\in\T_h} h_T$ denotes the mesh size.

Now we introduce a WG scheme for the eigenvalue problem
(\ref{problem-eq}).
 For a given integer $k\ge 1$, define the WG finite
element space
\begin{eqnarray*}
V_h=\{v=(v_0,v_b):v_0|_T\in P_k(T), v_b|_e\in P_{k-1}(e), \forall
T\in\T_h, e\in\E_h, \text{ and }v_b=0 \text{ on }\partial\Omega\}.
\end{eqnarray*}
For each weak function $v\in V_h$, we can define its weak gradient
$\nabla_w v$ by distribution element-wisely as follows.
\begin{definition}\cite{Wang2013a}
For each $v\in V_h$, $\nabla_w v|_T$ is the unique polynomial in
$[P_{k-1}(T)]^d$ satisfying
\begin{eqnarray}\label{def-wgradient}
(\nabla_w v,\bq)_T=-(v_0,\nabla\cdot\bq)_T+\langle v_b,\bq\cdot\bn
\rangle_{\partial T},\quad\forall\bq\in [P_{k-1}(T)]^d,
\end{eqnarray}
where $\bn$ denotes the outward unit normal vector.
\end{definition}

For the aim of analysis, some projection operators are also employed
in this paper. Let $Q_0$ denote the $L^2$ projection from $L^2(T)$
onto $P_k(T)$, $Q_b$ denote the $L^2$ projection from $L^2(e)$ onto
$P_{k-1}(e)$, and $\dQ_h$ denote the $L^2$ projection from
$[L^2(T)]^d$ onto $[P_{k-1}(T)]^d$. Combining $Q_0$ and $Q_b$
together, we can define $Q_h=\{Q_0,Q_b\}$, which is a projection
from $H_0^1(\Omega)$ onto $V_h$.

Now we define three bilinear forms on $V_h$ for any $v,w\in V_h$,
\begin{eqnarray*}
s(v,w)&=&\sumT h_T^{-1+\eps}\langle Q_b v_0-v_b,Q_b
w_0-w_b\rangle_{\partial T},\\
a_s(v,w)&=&(\nabla_w v,\nabla_w w)+s(v,w),\\
b_w(v,w)&=&(v_0,w_0),
\end{eqnarray*}
where $0\leq \eps<1$ is a constant \cite{Xie2015}.
Define the following norm on $V_h$ that
\begin{eqnarray*}
\trb v^2=a_s(v,v),\quad\forall v\in V_h.
\end{eqnarray*}
For the simplicity of notation, we introduce a semi-norm
$\|\cdot\|_{b}$ by $$\|v\|^2_b:=b_w(v, v),\quad\forall v\in V_h.$$
With these preparations we can give the following WG algorithm.
\begin{algorithm1}\cite{Xie2015}
Find $(\lambda_h, u_h)\in\Real\times V_h$ such that $\|u_h\|_b=1$ and
\begin{eqnarray}\label{WG-scheme}
a_s(u_h,v_h)=\lambda_h b_w(u_h,v_h),\quad\forall v_h\in V_h.
\end{eqnarray}
\end{algorithm1}
Denote $V_0=H_0^1(\Omega)$, and define the sum space $V=V_0+V_h$.
Now we introduce the following semi-norm on $V$ that
\begin{eqnarray*}
\|w\|_V^2= \sumT \Big(\|\nabla w_0\|_T^2+ h_T^{-1}\|Q_b
w_0-w_b\|^2_{\partial T}\Big).
\end{eqnarray*}
$\|\cdot\|_V$ indeed defines a norm on $V$ \cite{Xie2015}.
{\color{black} For the analysis in this paper, we still need
to introduce the dual norm of $\|\cdot\|_{V}$ as follows
\begin{eqnarray*}
\|v_h\|_{-V}=\sup_{w\in V, \|w\|_V\ne 0}\frac{b_w(v_h,w)}{\|w\|_V}.
\end{eqnarray*}}
For the standard WG scheme, the following convergence theorem
holds true, and which also gives a lower bound estimate.
\begin{theorem}\cite{Xie2015}\label{Error_Theorem_Eigenpair}
{\color{black} Suppose $\lambda_{j,h}$ is the $j$-th eigenvalue of
(\ref{WG-scheme}) and $u_{j,h}$ is the corresponding eigenfunction.
There exists an exact eigenfunction $u_j$ corresponding to the $j$-th exact eigenvalue $\lambda_j$
such that the following error estimates hold}
\begin{eqnarray}\label{eig-est1}
&&Ch^{2k}\|u_j\|_{k+1}\le\lambda_j-\lambda_{j,h}\le
Ch^{2k-2\eps}\|u_j\|_{k+1},
\\\label{eig-est2}
&&\|u_j-u_{j,h}\|_V \le Ch^{k-\eps}\|u_j\|_{k+1},
\\\label{eig-est3}
&&\|u_j-u_{j,h}\|_b\le Ch^{k+1-\eps}\|u_j\|_{k+1},
\end{eqnarray}
when $u_j\in H^{k+1}(\Omega)$ and $h$ is small enough.
\end{theorem}

\section{{\color{black}Error estimate in negative norm}}
In this section, we shall analysis the $\|\cdot\|_{-V}$ error estimate for the
WG scheme (\ref{WG-scheme}). First, we need to establish the $\|\cdot\|_{-V}$ error
estimate for the corresponding boundary value problem.
Consider the Poisson equation
\begin{equation}\label{problem-poisson}
\left\{
\begin{array}{rcl}
-\Delta u &=& f,\quad \text{in }\Omega,\\
u &=& 0,\quad \text{on }\partial\Omega,
\end{array}
\right.
\end{equation}
where $\Omega$ is a polygon or polyhedra in $\Real^d$ $(d=2,3)$.

The WG method is adopted to solve equation (\ref{problem-poisson}).
For analysis, we define the following norm
\begin{eqnarray*}
\trb{v}_{-1}=\sup_{w\in V_h,w\ne 0}\frac{b_w(v,w)}{\trb w}.
\end{eqnarray*}
It is easy to check that $\|\cdot\|_V$ is equivalent to $\|\cdot\|_1$ on the space
$H_0^1(\Omega)$. The relationship between $\|\cdot\|_V$ and $\trb\cdot$
has been discussed in \cite{Xie2015}, which is presented as
follows.
\begin{lemma}\cite{Xie2015}\label{norm-equi}
There exist two constants $C_1$ and $C_2$ such that the following inequalities hold
for any $w\in V_h$
\begin{eqnarray}\label{norm-equi2}
C_1\trb w\le \|w\|_V \le C_2h^{-\frac\eps 2}\trb w.
\end{eqnarray}
\end{lemma}
The WG method for the boundary value problem (\ref{problem-poisson}) can be described as follows:
\begin{algorithm1}
Find $u_h\in V_h$ such that
\begin{eqnarray}\label{WG-scheme1}
a_s(u_h,v)=b_w(f,v),\quad\forall v\in V_h.
\end{eqnarray}
\end{algorithm1}
Suppose $u$ is the exact solution for (\ref{problem-poisson}) and
$u_h$ is the corresponding numerical solution of (\ref{WG-scheme1}). Denote by $e_h$
the error that
\begin{eqnarray*}
e_h=Q_h u-u_h=\{Q_0 u-u_0,Q_b u-u_b\}.
\end{eqnarray*}
Then $e_h$ satisfies the following equation.
\begin{lemma}\cite{Xie2015}\label{err-eq}
Let $e_h$ be the error of the weak Galerkin scheme (\ref{WG-scheme1}).
Then we have
\begin{eqnarray}\label{t1}
a_s(e_h,v)=\ell(u,v),\ \ \ \ \forall v\in V_h,
\end{eqnarray}
where
\begin{eqnarray*}
\ell(u,v)=\sumT\langle(\nabla u-\dQ_h\nabla u)\cdot\bn, v_0-v_b\rangle_{\partial T}+s(Q_h u,v).
\end{eqnarray*}
Moreover, we have
\begin{eqnarray}
a_s(Q_hu,v)=\ell(u,v)+b_w(f,v), \ \ \ \ \forall v\in V_h.
\end{eqnarray}
\end{lemma}
\begin{theorem}\cite{Xie2015}\label{err-est}
Assume the exact solution $u$ of (\ref{problem-poisson}) satisfies
$u\in H^{k+1}(\Omega)$ and $u_h$ is the numerical solution of the WG scheme (\ref{WG-scheme1}).
Then the following error estimate holds true,
\begin{eqnarray}\label{err-est1}
\trb{Q_h u-u_h}\le Ch^{k-\frac\eps 2}\|u\|_{k+1}.
\end{eqnarray}
\end{theorem}
Now, we come to estimate the error $e_h$ in the norm $\trb\cdot_{-1}$.
{\color{black}
We suppose the partition $\T_h$ is a triangulation, instead of an arbitrary
polytopal mesh.}
The idea is to introduce a continuous interpolation for
$v_h\in V_h$. To this end, we define $\mathcal N_T$ as the vertices of the
element $T\in \T_h$.
Here, the notation $V_h^C$ is used to denote the conforming linear finite element
space
\cite{Brenner2008,MR1930132}. We need to define an interpolation operator
$\Pi_h: V_h\rightarrow V_h^C\subset V_0$ as follows.
For each node $A$ in $\T_h$, let $$K(A):=\bigcup_{A\in \mathcal N_T}T$$
and $N_A$ is the number of elements in $K(A)$.
Then, for any $v_h\in V_h$, the value of $\Pi_h v_h$ at the node $A$ is defined by
\begin{eqnarray*}
(\Pi_h v_h)(A)=\frac{1}{N_A}\sum_{T\in K(A)} v_0|_T(A).
\end{eqnarray*}
Then the function $\Pi_h v_h\in V_h^C$ is determined by its nodal values and
the basis for the space $V_h^C$.
\begin{lemma}\label{technique-1}
For any $v_h\in V_h$, we have the following estimate
\begin{eqnarray}\label{Boundness_H_1}
|\Pi_h v_h|_1&\lesssim &\|v_h\|_V.
\end{eqnarray}
\end{lemma}
\begin{proof}
For any $T\in\T_h$, define
$$K(T):=\bigcup_{\mathcal N_T\cap \mathcal N_T'\neq\varnothing}T'.$$
We only need to prove that
\begin{eqnarray}\label{est1}
|\Pi_h v_h|_{1,T}\lesssim \|v_h\|_{V,K(T)}
\end{eqnarray}
since summing (\ref{est1}) over $T\in\T_h$ can lead to the desired result (\ref{Boundness_H_1}).

Define $\widehat T$ the reference element and $F_T:T\rightarrow \widehat T$ the affine isomorphism.
Denote $K(\widehat T)=F_T(K(T))$. It follows from the regularity assumption of the mesh that
$K(\widehat T)$ is also of unit size. Then we define the following Banach spaces
$$\widehat V_h=\Big\{\widehat v_h:\ \widehat v_h=v_h(F_T^{-1}(\widehat x))=v_h\circ F_T^{-1}, \ \
\forall v_h\in V_h|_{K(T)}\Big\}$$
and $M=\big\{\widehat v_h:\widehat v_h\in \widehat V_h\ {\rm and}\ \int_{\widehat T}
\widehat v_hd\widehat T=0\big\}$. Obviously the complement of $M$ in $\widehat V_h$
with the $L^2$ inner product is
$M^\perp=\big\{\widehat v_h:\widehat v_h\in \widehat V_h\ {\rm and}\  \widehat v_h\
{\rm is\ a\ constant\ on}\ K(\widehat T)\big\}$.
We also define the interpolation operator $\widehat\Pi_T:=\Pi_h\circ F_T^{-1}$ on $K(\widehat T)$
corresponding the operator $\Pi_h$ on $K(T)$.
Notice that $|\widehat\Pi_T \widehat v_h|_{1,\widehat T}$ defines a seminorm on $M$ and
$\|\widehat v_h\|_{V,K(\widehat T)}$ defines a norm on $M$. From the equivalence of norms on finite
dimensional Banach spaces, we obtain
\begin{eqnarray*}
|\widehat\Pi_T \widehat v_h|_{1,\widehat T}\lesssim \|\widehat v_h\|_{V,K(\widehat T)},
\quad\forall \widehat v_h\in M.
\end{eqnarray*}
Furthermore, since $|\widehat\Pi_T \widehat v_h|_{1,\widehat T}=\|\widehat v_h\|_{V,K(\widehat T)}=0$
for all $\widehat v_h \in M^\perp$, we have
\begin{eqnarray*}
|\widehat\Pi_T \widehat v_h|_{1,\widehat T}\lesssim \|\widehat v_h\|_{V,K(\widehat T)},
\quad\forall \widehat v_h\in \widehat V_h.
\end{eqnarray*}
From the property of affine isomorphism, the following inequalities hold
\begin{eqnarray*}
|\Pi_h v_h|_{1,T}&\lesssim & h_T^{\frac{d}{2}-1}|\widehat\Pi_T \widehat v_h|_{1,\widehat T}
\lesssim  h_T^{\frac{d}{2}-1}\|\widehat v_h\|_{V,K(\widehat T)}
\lesssim \|v_h\|_{V,K(T)}.
\end{eqnarray*}
Then the proof is completed.
\end{proof}
\begin{lemma}\label{technique-3}
For any $v_h\in V_h$, we have the following estimate
\begin{eqnarray*}
\|v_h-\Pi_h v_h\|_b&\lesssim &h\|v_h\|_V.
\end{eqnarray*}
\end{lemma}
\begin{proof}
Similarly to the proof of Lemma \ref{technique-1}, we only need to prove that
\begin{eqnarray}\label{est4}
\|v_0-\Pi_h v_h\|_{T}\lesssim h\|v_h\|_{V,K(T)},\ \ \ \ \forall T\in\T_h.
\end{eqnarray}
First, on the element $T$, we have the following estimates
\begin{eqnarray}\label{Inequality_1}
\|\Pi_{1,T}v_0-\Pi_h v_h\|_{T}^2
&=&\int_T\sum_{A_i\in\mathcal N_T}|v_0(A_i)-\Pi_h v_h(A_i)|^2\varphi_i^2 dT\nonumber\\
&\lesssim&h^d \sum_{A_i\in\mathcal N_T}|v_0(A_i)-\Pi_h v_h(A_i)|^2,
\end{eqnarray}
where $\varphi_i$ is the linear Lagrange basis function corresponding to $A_{i}$
and $\Pi_{1,T}$ is the linear Lagrange interpolation for the finite element space $V_h^C$
on the element $T$.

For each node $A_i$, denote $\{T_1,T_2,\cdots,T_{N_i}\}$  the elements
in $K(A_i)$ in counter-clock order. From the definition of
$\Pi_h v_h$ we can obtain
\begin{eqnarray}\label{Inequality_2}
|v_0(A_i)-\Pi_h v(A_i)|
&\le&\dfrac{1}{N_i}\sum_{j=1}^{N_i}\Big|v_0|_T(A_i)-v_0|_{T_j}(A_i)\Big|\nonumber\\
&\le&\dfrac{1}{N_i}\sum_{j=1}^{N_i}\sum_{k=1}^j\Big|v_0|_{T_k}(A_i)-v_0|_{T_{k-1}}(A_i)\Big|.
\end{eqnarray}
From the $L^\infty$-$L^2$ inverse inequality, it follows that
\begin{eqnarray}\label{Inequality_3}
\Big|v_0|_{T_{k}}(A_i)-v_0|_{T_{k-1}}(A_i)\Big|
\lesssim h^{-\frac{d-1}{2}}\|[\![v_0]\!]\|_e,
\end{eqnarray}
where $e$ is the edge between $T_{k}$ and $T_{k-1}$.

Combining (\ref{Inequality_1})-(\ref{Inequality_3}) and the definition of the norm $\|\cdot\|_V$
leads to the following estimates
\begin{eqnarray*}
&&\|\Pi_{1,T}v_0-\Pi_h v_h\|_{T}^2\lesssim  h\sum_{e\in K(T)}\|[\![v_0]\!]\|_e^2
\lesssim h\sum_{T\in K(T)}\|v_0-v_b\|_{\partial T}^2
\lesssim h^2\|v_h\|_{V,K(T)}^2.
\end{eqnarray*}
Together with $\|v_0-\Pi_{1,T}v_0\|_T^2\lesssim h^2\|\nabla v_0\|_T^2$, we can
obtain the desired result (\ref{est4}) easily and the proof is completed.
\end{proof}
\begin{lemma}\label{Vh-V-equi}
For any $\varphi\in V_h$, there exists $\widetilde\varphi\in V_0$ such that
\begin{eqnarray}\label{Auxiliarly_Estimate}
\|\widetilde \varphi\|_1\lesssim \|\varphi\|_V\ \ \ \ \ {\rm and}\ \ \ \ \
\|\widetilde \varphi-\varphi\|_b\lesssim h\|\widetilde\varphi\|_1.
\end{eqnarray}
\end{lemma}
The proof can be given easily by combining Lemmas \ref{technique-1}, Lemma \ref{technique-3}, and
taking $\widetilde\varphi=\Pi_h\varphi$ which is a function in $V_0$.

In order to deduce the error estimate in $\|\cdot\|_{-V}$, we define the following dual problem
\begin{equation}\label{Dual_Problem}
\left\{
\begin{array}{rcl}
-\Delta \psi&=&\widetilde\varphi,\quad\text{ in }\Omega,\\
\psi&=&0,\quad\ \text{on }\partial\Omega,
\end{array}
\right.
\end{equation}
where $\widetilde\varphi\in V_0$.
\begin{theorem}\label{err-estH-1}
Assume $u\in H^{k+1}(\Omega)$ is the exact solution of (\ref{problem-poisson})
and $u_h$ is the numerical solution of the WG scheme (\ref{WG-scheme1}).
If the solution $\psi$ of the dual problem (\ref{Dual_Problem}) has
$H^3(\Omega)$-regularity and $k\geq 2$, the following estimate holds true
\begin{eqnarray}\label{H-1_Estimate_Solution}
\trb{Q_h u-u_h}_{-1}\le Ch^{k+2-\eps}\|u\|_{k+1}.
\end{eqnarray}
\end{theorem}
\begin{proof}
Denote $e_h=Q_hu-u_h$. We choose $\phi\in V_h$ and $\widetilde\phi\in V_0$ such that
$\trb{\phi}=1$, $\trb{e_h}=b_w(e_h,\phi)$, and $\widetilde\phi$ satisfies the estimates
in (\ref{Auxiliarly_Estimate}). From Lemma \ref{err-eq}, we have
\begin{eqnarray}\label{t2}
a_s(Q_h \psi,v)=\ell(\psi,v)+b_w(\varphi,v),\ \ \ \forall v\in V_h.
\end{eqnarray}
Taking $v=Q_h\psi$ in (\ref{t1}) and $v=e_h$ in (\ref{t2}),
and subtracting (\ref{t1}) from (\ref{t2}), we have
\begin{eqnarray*}
(e_0,\widetilde\varphi)=\ell(u,Q_h\psi)-\ell(\psi,e_h).
\end{eqnarray*}
Since $\psi\in H^3(\Omega)$, $k>2$, and $u\in H^{k+1}(\Omega)$, the following estimates hold
\begin{eqnarray}\label{Estimate_1}
\ell(\psi,e_h)&=&\sumT\langle(\nabla \psi-\dQ_h\nabla \psi)\cdot\bn,
e_0-e_b\rangle_{\partial T}\nonumber\\
&&+\sumT h_T^{-1+\eps}\langle Q_bQ_0 \psi-Q_b \psi,Q_b e_0-e_b
\rangle_{\partial T}\nonumber\\
&\le& Ch^{2-\frac\eps 2}\|\psi\|_3\trb{e_h},
\end{eqnarray}
\begin{eqnarray}\label{Estimate_2}
\ell(u,Q_h \psi)&=&\sumT\langle(\nabla u-\dQ_h\nabla u)\cdot\bn,
Q_0\psi-Q_b\psi\rangle_{\partial T}\nonumber\\
&&+\sumT h_T^{-1+\eps}\langle Q_bQ_0 u-Q_b u,Q_b Q_0\psi-Q_b\psi
\rangle_{\partial T}\nonumber\\
&\le& Ch^{k+2-\eps}\|\psi\|_3\|u\|_{k+1}.
\end{eqnarray}
Thus, combining (\ref{Estimate_1})-(\ref{Estimate_2}) and Lemma \ref{Vh-V-equi} leads to
\begin{eqnarray*}
\trb{e_h}_{-1}&=&b(e_h,\varphi)\le(e_0,\tilde\varphi)+\|e_0\|\|\varphi-\tilde\varphi\|\\
&\le& Ch^{k+2-\eps}\|\tilde\varphi\|_1\|u\|_{k+1}\\
&\le& Ch^{k+2-\eps}\|u\|_{k+1},
\end{eqnarray*}
which completes the proof.
\end{proof}

From Lemma \ref{norm-equi} and Theorem \ref{err-estH-1}, we have the following corollary.
\begin{corollary}\label{Auxiliary_Corollary}
Under the conditions of Theorem \ref{err-estH-1}, the following estimate holds true
\begin{eqnarray}
\|{Q_h u-u_h}\|_{-V}\le Ch^{k+2-\eps}\|u\|_{k+1}.
\end{eqnarray}
\end{corollary}
Here, we shall also give the estimate for the projection error $\|u-Q_h u\|_{-V}$.
\begin{lemma}\label{proj-V}
When $u\in H^{k+1}(\Omega)$, the following estimate holds true
\begin{eqnarray*}
\|{u-Q_h u}\|_{-V}\le Ch^{k+2}\|u\|_{k+1}.
\end{eqnarray*}
\end{lemma}
\begin{proof}
From the definition, we know there exists $v\in V$ such that
\begin{eqnarray*}
\|u-Q_h u\|_{-V}=\frac{b_w(u-Q_h u,v)}{\|v\|_V}.
\end{eqnarray*}
Since $V=H_0^1(\Omega)\oplus(V_h\backslash H_0^1(\Omega))$, $v$ can be
decomposed as $v=v_1+v_2$, where $v_1\in  H_0^1(\Omega)$, $v_2\in
V_h\backslash H_0^1(\Omega)$.
It follows that
\begin{eqnarray*}
\|u-Q_h u\|_{-V}&=&\frac{b_w(u-Q_h u,v)}{\|v\|_V}\\
&=&\frac{b_w(u-Q_h u,v_1)}{\|v\|_V}+\frac{b_w(u-Q_h u,v_2)}{\|v\|_V}\\
&\le& \frac{b_w(u-Q_h u,v_1)}{\|v_1\|_V}\\
&\le& C\frac{b_w(u-Q_h u,v_1-Q_h v_1)}{\|v_1\|_1}\\
&\le& Ch^{k+2}\|u\|_{k+1},
\end{eqnarray*}
where we used the following error estimates for the projection operator $Q_h$
\begin{eqnarray*}
\|u-Q_hu\|_b\leq Ch^{k+1}\|u\|_{k+1} \ \ \ {\rm and}\ \ \ \|v_1-Q_hv_1\|_b\leq Ch\|v_1\|_1.
\end{eqnarray*}
Then the proof is completed.
\end{proof}

Combining Corollary \ref{Auxiliary_Corollary} with Lemma \ref{proj-V},
we have the following error estimate result for the boundary value problem (\ref{problem-poisson}).
\begin{theorem}\label{err-est-V}
Under the conditions of Theorem \ref{err-estH-1}, the following estimate holds true
\begin{eqnarray}
\|u-u_h\|_{-V}\le Ch^{k+2-\eps}\|u\|_{k+1}.
\end{eqnarray}
\end{theorem}
From the Babu\v{s}ka's theory and the results in \cite{Xie2015},
the conclusion of Theorem \ref{err-est-V} can be extended to the eigenvalue problem which means we have
the following error estimate and the proof is similar to \cite[Section 4]{Xie2015}.

{\color{black}
\begin{theorem}\label{Error_Theorem_Dualnorm}
Suppose $\lambda_{j,h}$ is the $j$-th eigenvalue of (\ref{WG-scheme})
and $u_{j,h}$ is the corresponding eigenfunction. Then there exists an exact
eigenfunction $u_j$ corresponding to the $j$-th exact eigenvalue of (\ref{problem-eq})
such that the following error estimate holds
\begin{eqnarray}\label{eig-est4}
&&\|u_j-u_{j,h}\|_{-V}\le Ch^{k+2-\eps}\|u_j\|_{k+1},
\end{eqnarray}
where $u_j\in H^{k+1}(\Omega)$, $k\geq 2$, and $h$ is sufficiently small.
\end{theorem}
}
\section{A two-grid scheme}
In this section, we propose a two-grid WG scheme for the
eigenvalue problem, and give the corresponding analysis for the
convergence and efficiency of this scheme. Here, we drop
the subscript $j$ to denote a certain eigenvalue of problem
(\ref{problem-eq}).

\begin{algorithm1}\label{twogrid-alg}

Step 1: Generate a coarse grid $\T_H$ on the domain $\Omega$ and
solve the following eigenvalue problem on the coarse grid $\T_H$:

 Find
$(\lambda_H,u_H)\in \Real\times V_H$, such that
\begin{eqnarray*}
a_s(u_H,v_H)=\lambda_H b_w(u_H,v_H),\quad\forall v_H\in V_H.
\end{eqnarray*}

Step 2: Refine the coarse grid $\T_H$ to obtain a finer grid $\T_h$
and solve one single linear problem on the fine grid $\T_h$:

Find
$\tilde u_h\in V_h$ such that
\begin{eqnarray*}
a_s(\tilde u_h,v_h)=\lambda_H b_w(u_H,v_h),\quad\forall v_h\in V_h.
\end{eqnarray*}

Step 3: Calculate the Rayleigh quotient for $\tilde u_{h}$
\begin{eqnarray*}
\tilde\lambda_h=\dfrac{a_s(\tilde u_h,\tilde u_h)}{b_w(\tilde
u_h,\tilde u_h)}.
\end{eqnarray*}
Finally, we obtain the eigenpair approximation $(\tilde\lambda_h,\tilde u_h)$.
\end{algorithm1}

First, we need the following discrete Poincar\'{e}'s inequality for the
WG method, which has been proved in \cite{Wang2014a}.
\begin{lemma}\label{discrete-poincare}
The discrete Poincar\'{e}-type inequality holds true on $V_h$, i.e.
\begin{eqnarray*}
\|v_h\|\lesssim \trb{v_h},\quad\forall v_h\in V_h.
\end{eqnarray*}
\end{lemma}
From Theorem \ref{Error_Theorem_Eigenpair}, suppose the
eigenfunction $u$ is smooth enough and we have the following
estimate immediately,
\begin{eqnarray*}
&&h^{2k}\lesssim \lambda-\lambda_h \lesssim h^{2k-2\eps}.
\end{eqnarray*}
For simplicity, here and hereafter, we assume the concerned
eigenvalues are simple. In order to estimate $|\lambda-\tilde\lambda_h|$,
we just need to estimate $|\lambda_h-\tilde\lambda_h|$.
\begin{lemma}\label{lemma1}
Suppose $(\tilde\lambda_h,\tilde u_h)$ is calculated by Algorithm
\ref{twogrid-alg} and $(\lambda_h, u_h)$ satisfies (\ref{WG-scheme}).
Then the following estimate holds
\begin{eqnarray}\label{Difference_twogrid_direct_eigenfun}
|\tilde\lambda_h-\lambda_h|\lesssim \trb{\tilde u_h- u_h}^2.
\end{eqnarray}
\end{lemma}
\begin{proof}
From (\ref{WG-scheme}) and Lemma \ref{discrete-poincare}, we have
\begin{eqnarray*}
&&(\tilde\lambda_h-\lambda_h) b_w(\tilde u_h,\tilde u_h)
=a_s(\tilde u_h,\tilde u_h)-\lambda_h b_w(\tilde u_h,\tilde u_h)\\
&=&a_s(\tilde u_h-u_h,\tilde u_h-u_h)+2a_s(u_h,\tilde u_h)-a_s(u_h,u_h)\\
&&-\lambda_hb_w(\tilde u_h-u_h,\tilde u_h-u_h)
-2\lambda_hb_w(u_h,\tilde u_h) +\lambda_hb_w(u_h,u_h)\\
&=&a_s(\tilde u_h-u_h,\tilde u_h-u_h)-\lambda_h b_w(\tilde u_h-u_h,
\tilde u_h-u_h)\\
&\lesssim& \trb{\tilde u_h- u_h}^2,
\end{eqnarray*}
which completes the proof.
\end{proof}

\begin{lemma}\label{lemma2}
Under the conditions of Lemma \ref{lemma1}, the following estimate holds true
\begin{eqnarray}\label{Difference_twogrid_direct_eigenfunction}
\trb{\tilde u_h-u_h}\lesssim H^{2k-2\eps}+{\color{black}H^{k+\gamma-\eps}}, \quad when \ h<H.
\end{eqnarray}
Here and hereafter $\gamma$ is defined as follows
\begin{equation}\label{Definition_Gamma}
\gamma=\left\{
\begin{array}{rl}
1,&\ {\rm when\ the\ solution\ of\ dual\ problem\ (\ref{Dual_Problem})\ satisfies}\\
&\ \ \ \psi\in H^2(\Omega)\ {\rm or}\ k=1,\\
2-\frac{\varepsilon}{2},&\ {\rm when\ the\ solution\ of\ dual\ problem\ (\ref{Dual_Problem})\
satisfies}\\
&\ \ \ \psi\in H^3(\Omega)\ {\rm and}\ k>1.
\end{array}
\right.
\end{equation}
\end{lemma}
\begin{proof}
For all $v_h\in V_h$, from equation (\ref{WG-scheme}), Theorems \ref{Error_Theorem_Eigenpair},
and \ref{Error_Theorem_Dualnorm} we can obtain
\begin{eqnarray*}\label{Inequality}
&&a_s(\tilde u_h-u_h,v_h)=a_s(\tilde u_h,v_h)-a_s(u_h,v_h)\nonumber\\
&=&\lambda_Hb_w(u_H,v_h)-\lambda_h b_w(u_h,v_h)\nonumber\\
&=&\lambda_Hb_w(u_H,v_h)-\lambda_Hb_w(u_h,v_h)+\lambda_Hb_w(u_h,v_h)
-\lambda_hb_w(u_h,v_h)\nonumber\\
&=&\lambda_Hb_w(u_H-u_h,v_h)+(\lambda_H-\lambda_h)b_w(u_h,v_h)\nonumber\\
&=&\lambda_Hb_w(u_H-u,v_h)+\lambda_Hb_w(u-u_h,v_h)\nonumber\\
&&+(\lambda_H-\lambda)b_w(u_h,v_h)+(\lambda-\lambda_h)b_w(u_h,v_h).
\end{eqnarray*}
If $k=1$ or the solution of the dual problem (\ref{Dual_Problem}) has the regularity
$\psi\in H^2(\Omega)$, we have
\begin{eqnarray}\label{Inequality_1_twogrid}
a_s(\tilde u_h-u_h,v_h)
&\lesssim& (\|u-u_H\|+\|u-u_h\|)\|v_h\|_b
+(|\lambda_H-\lambda|+|\lambda_h-\lambda|)\|v_h\|_b\nonumber\\
&\lesssim& (H^{{\color{black}k+1}-\eps}+h^{{\color{black}k+1}-\eps})\|v_h\|_b
+(H^{2k-2\eps}+h^{2k-2\eps})\|v_h\|_b\nonumber\\
&\lesssim& (H^{{\color{black}k+1}-\eps}+H^{2k-2\eps})\trb{v_h}.
\end{eqnarray}
If $k>1$ and the solution of the dual problem (\ref{Dual_Problem}) has the regularity
$\psi\in H^3(\Omega)$, the following estimates hold
\begin{eqnarray}\label{Inequality_2_twogrid}
a_s(\tilde u_h-u_h,v_h)
&\lesssim& (\|u-u_H\|_{-V}+\|u-u_h\|_{-V})\|v_h\|_V
+(|\lambda_H-\lambda|+|\lambda_h-\lambda|)\|v_h\|_b\nonumber\\
&\lesssim& (H^{{\color{black}k+2}-\eps}+h^{{\color{black}k+2}-\eps})\|v_h\|_V
+(H^{2k-2\eps}+h^{2k-2\eps})\|v_h\|\nonumber\\
&\lesssim& (H^{{\color{black}k+2}-\frac32\eps}+H^{2k-2\eps})\trb{v_h}.
\end{eqnarray}
From (\ref{Inequality_1_twogrid})-(\ref{Inequality_2_twogrid}) and taking $v_h=\tilde u_h-u_h$,
we can obtain the desired result (\ref{Difference_twogrid_direct_eigenfunction})
and the proof is completed.
\end{proof}

From Lemmas \ref{lemma1} and \ref{lemma2}, the convergence
of $|\lambda_h-\tilde \lambda_h|$ follows immediately.
\begin{lemma}\label{Two_Grid_Eigenvalue_Lemma}
Suppose $(\tilde\lambda_h,\tilde u_h)$ is calculated by Algorithm
\ref{twogrid-alg} and $(\lambda_h, u_h)$ satisfies (\ref{WG-scheme}).
Then the following estimate holds
\begin{eqnarray}\label{Difference_twogrid_solving_eigenvalue}
|\tilde\lambda_h-\lambda_h|\lesssim H^{4k-4\eps}+H^{{\color{black}2k+2\gamma-2\eps}}.
\end{eqnarray}
\end{lemma}
With Lemmas \ref{lemma2} and \ref{Two_Grid_Eigenvalue_Lemma},
we arrive at the following convergence theorem.
\begin{theorem}\label{Error_Eigen_Two_Grid_Theorem}
Suppose $(\tilde\lambda_h,\tilde u_h)$ is calculated by
Algorithm \ref{twogrid-alg}, $h<H$ and the exact eigenfunctions of (\ref{problem-eq})
have $H^{k+1}(\Omega)$-regularity.
Then there exists an exact eigenpair $(\lambda, u)$ such that the following estimates hold true
\begin{eqnarray}
\trb{Q_hu-\tilde u_h} &\lesssim& H^{\bar k}+h^{k-\eps},\label{Error_Eigenfunction_Two_Grid}\\
|\tilde\lambda_h-\lambda|&\lesssim& H^{2\bar k}+h^{2k-2\eps},\label{Error_Eigen_Two_Grid}
\end{eqnarray}
where $\bar k=\min\{2k-2\eps, {\color{black}k+\gamma-\eps}\}$.
\end{theorem}

From Theorem \ref{Error_Theorem_Eigenpair} and Lemma
\ref{Two_Grid_Eigenvalue_Lemma}, we can get the following lower
bound estimate.
\begin{theorem}\label{Lower_Bound_Theorem_TwoGrid}
Suppose the conditions of Theorem \ref{Error_Eigen_Two_Grid_Theorem} hold and
let $\bar k=\min\{2k-2\eps, k+\gamma-\eps\}$ and $\delta$ be a positive number.
If $H^{2\bar k}\le C h^{2k+\delta}$, then we have
\begin{eqnarray*}
\lambda-\tilde \lambda_h\ge 0,
\end{eqnarray*}
when $H$ and $h$ are sufficiently small.
\end{theorem}
\begin{proof}
From Theorem \ref{Error_Theorem_Eigenpair} we have
\begin{eqnarray*}
\lambda-\lambda_h\ge Ch^{2k}.
\end{eqnarray*}
According to Lemma \ref{Two_Grid_Eigenvalue_Lemma}, the following estimates hold
\begin{eqnarray*}
|\tilde\lambda_h-\lambda_h|\le CH^{2\bar k}\le  Ch^{2k+\delta}.
\end{eqnarray*}
When $h$ is sufficiently small, it follows that
\begin{eqnarray*}
\lambda-\tilde\lambda_h
&=&\lambda-\lambda_h+\lambda_h-\tilde\lambda_h
\ge\lambda-\lambda_h-|\tilde\lambda_h-\lambda_h|\\
&\ge&Ch^{2k}-Ch^{2k+\delta}\ge 0,
\end{eqnarray*}
which completes the proof.
\end{proof}

\section{A two-space scheme}

In this section, we shall give a two-space WG scheme for problem
(\ref{problem-eq}), where different polynomial spaces are employed
on the same mesh.

Denote the finite element spaces
\begin{eqnarray*}
V_h^1=\{v=(v_0,v_b):v_0|_T\in P_{k_1}(T), v_b|_e\in P_{k_1-1}(e), \forall
T\in\T_h, e\in\E_h, \text{ and }v_b=0 \text{ on }\partial\Omega\},
\end{eqnarray*}
\begin{eqnarray*}
V_h^2=\{v=(v_0,v_b):v_0|_T\in P_{k_2}(T), v_b|_e\in P_{k_2-1}(e), \forall
T\in\T_h, e\in\E_h, \text{ and }v_b=0 \text{ on }\partial\Omega\}.
\end{eqnarray*}

Denote $(\lambda^2_h,u^2_h)$ the numerical solution of the
standard WG scheme, which satisfies
\begin{eqnarray}\label{tg-est2}
a_s(u^2_h,v_h)=\lambda^2_hb_w(u^2_h,v_h),\quad\forall v_h\in V^2_h.
\end{eqnarray}

In the two-space method, the spaces $V_h^1$ and $V_h^2$ are defined on the same mesh, with
different degrees of polynomials.

\begin{algorithm1}\label{twospace-alg}

Step 1: Solve the eigenvalue problem in the space $V_h^1$:

Find $(\lambda_h^1,u_h^1)\in \Real\times V_h^1$ such that
\begin{eqnarray*}
a_s(u_h^1,v_h)=\lambda_h^1b_w(u_h^1,v_h),\quad\forall v_h\in V_h^1.
\end{eqnarray*}

Step 2: Solve one single linear problem in the space $V_h^2$:

Find $\hat u_h\in V_h^2$ such that
\begin{eqnarray*}
a_s(\hat u_h,v_h)=\lambda_h^1b_w(u_h^1,v_h),\quad\forall v_h\in V_h^2.
\end{eqnarray*}

Step 3: Calculate the Rayleigh quotient
\begin{eqnarray*}
\hat \lambda_h=\dfrac{a_s(\hat u_h,\hat u_h)}{b_w(\hat u_h,\hat u_h)}.
\end{eqnarray*}
Finally, we obtain the eigenpair approximation $(\hat \lambda_h,\hat u_h)$.
\end{algorithm1}

The proof is similar to the two-grid algorithm, and we just need to
interpret Lemmas \ref{lemma1} and \ref{lemma2} into the two-space case.
\begin{lemma}\label{lemma4}
Suppose $(\hat\lambda_h,\hat u_h)$ is calculated by Algorithm
\ref{twospace-alg} and $(\lambda_h^2, u^2_h)$ satisfies (\ref{tg-est2}).
Then the following estimate holds
\begin{eqnarray}
|\hat\lambda_h-\lambda_h^2|\lesssim \trb{\hat u_h- u^2_h}^2.
\end{eqnarray}
\end{lemma}
\begin{lemma}\label{lemma3}
Suppose $(\hat\lambda_h,\hat u_h)$ is calculated by Algorithm \ref{twospace-alg},
$(\lambda^2_h, u^2_h)$ satisfies (\ref{tg-est2}) and $k_1<k_2$.
When the exact solution $u\in H^{k_2+1}(\Omega)$, the following estimate holds true
\begin{eqnarray}\label{Difference_twospace_direct_eigenfunction}
\trb{\hat u_h-u^2_h}\lesssim h^{2k_1-2\eps}+h^{{\color{black}k_1+\gamma-\eps}}.
\end{eqnarray}
\end{lemma}
\begin{proof}
For all $v_h\in V_h^2$, from formula (\ref{tg-est2}), Theorems
\ref{Error_Theorem_Eigenpair} and \ref{Error_Theorem_Dualnorm}, we can obtain
\begin{eqnarray*}
&&a_s(\hat u_h^2-u_h^2,v_h)=a_s(\hat u_h^2,v_h)-a_s(u_h^2,v_h)\\
&=&\lambda_h^1b_w(u_h^1 ,v_h)-\lambda_h^2b_w(u_h ^2,v_h)\\
&=&\lambda_h^1b_w(u_h^1 ,v_h)-\lambda_h^1b_w(u_h^2,v_h)
+\lambda_h^1b_w(u_h ^2,v_h) -\lambda_h^2b_w(u_h ^2,v_h)\\
&=&\lambda_h^1b_w(u_h^1 -u_h^2,v_h)+(\lambda_h^1-\lambda_h^2)b_w(u_h ^2,v_h)\\
&=&\lambda_h^1b_w(u_h^1 -u,v_h)+\lambda_h^1b_w(u-u_h ^2,v_h)\\
&&+(\lambda_h^1-\lambda)b_w(u_h^2,v_h)+(\lambda-\lambda_h^2)b_w(u_h^2,v_h).
\end{eqnarray*}
If $k_1=1$ or the solution of the dual problem (\ref{Dual_Problem}) has the regularity
$\psi\in H^2(\Omega)$, we have
\begin{eqnarray}\label{Inequality_1_twospace}
a_s(\hat u_h^2-u_h^2,v_h)
&\lesssim& (\|u-u_h^1\|+\|u-u_h^2\|)\|v_h\|_b
+(|\lambda_h^1-\lambda|+|\lambda_h^2-\lambda|)\|v_h\|_b\nonumber\\
&\lesssim&(h^{{\color{black}k_1+1}-\eps}+h^{{\color{black}k_2+1}-\eps})\|v_h\|_b
+(h^{2k_1-2\eps}+h^{2k_2-2\eps})\|v_h\|_b\nonumber\\
&\lesssim& (h^{2k_1-2\eps}+h^{{\color{black}k_1+1}-\eps})\trb{v_h}.
\end{eqnarray}
If $k_1>1$ and the solution of the dual problem (\ref{Dual_Problem}) has the regularity
$\psi\in H^3(\Omega)$, the following estimates hold
\begin{eqnarray}\label{Inequality_2_twospace}
a_s(\hat u_h^2-u_h^2,v_h)
&\lesssim& (\|u-u_h^1\|_{-V}+\|u-u_h^2\|_{-V})\|v_h\|_V
+(|\lambda_h^1-\lambda|+|\lambda_h^2-\lambda|)\|v_h\|_b\nonumber\\
&\lesssim&(h^{{\color{black}k_1+2}-\eps}+h^{{\color{black}k_2+2}-\eps})\|v_h\|_V
+(h^{2k_1-2\eps}+h^{2k_2-2\eps})\|v_h\|\nonumber\\
&\lesssim& (h^{2k_1-2\eps}+h^{{\color{black}k_1+2}-\frac32\eps})\trb{v_h}.
\end{eqnarray}
From (\ref{Inequality_1_twospace})-(\ref{Inequality_2_twospace}) and taking $v_h=\hat u_h-u^2_h$,
we can obtain the desired result (\ref{Difference_twospace_direct_eigenfunction})
and the proof is completed.
\end{proof}

With Lemmas \ref{lemma4} and \ref{lemma3}, we can get the following
estimate easily.
\begin{theorem}\label{Two_Space_Thm}
Suppose $(\hat\lambda_h,\hat u_h)$ is calculated by Algorithm
\ref{twospace-alg}, $(\lambda, u)$ satisfies (\ref{problem-eq}) and
$k_1\le k_2$. When the exact solution $u\in H^{k_2+1}(\Omega)$,
the following estimate holds true
\begin{eqnarray*}
|\hat\lambda_h-\lambda|\lesssim h^{\hat k}+h^{2k_2-2\eps},
\end{eqnarray*}
where $\hat k=\min\{4k_1-4\eps,{\color{black}2k_1+2\gamma-2\eps}\}$.
\end{theorem}

From Theorem \ref{Error_Theorem_Eigenpair} and Lemma \ref{lemma3},
we can get the following lower bound estimate.
\begin{theorem}\label{Lower_Bound_Theorem_TwoSpace}
Suppose the assumptions of Theorem \ref{Two_Space_Thm} hold true and let $\hat
k=\min\{4k_1-4\eps,2k_1+2\gamma-2\eps\}$ and $\delta$ be a positive number.
If $\hat k> 2k_2$, we have
\begin{eqnarray*}
\lambda-\hat \lambda_h\ge 0
\end{eqnarray*}
when $h$ is sufficiently small.
\end{theorem}
\begin{proof}
From Theorem \ref{Error_Theorem_Eigenpair}, we have
\begin{eqnarray*}
\lambda-\lambda_h\ge Ch^{2k_2}.
\end{eqnarray*}
According to Lemma \ref{lemma3}, the following inequalities hold
\begin{eqnarray*}
|\hat\lambda_h-\lambda_h|\le Ch^{\hat k}\le  Ch^{2k_2+\delta}.
\end{eqnarray*}
When $h$ is sufficiently small, it follows that
\begin{eqnarray*}
\lambda-\hat\lambda_h&=&\lambda-\lambda_h+\lambda_h-\hat\lambda_h
\ge\lambda-\lambda_h-|\hat\lambda_h-\lambda_h|\\
&\ge&Ch^{2k_2}-Ch^{2k_2+\delta}\ge 0,
\end{eqnarray*}
which completes the proof.
\end{proof}
\section{Numerical Experiments}
In this section, we present two numerical examples of Algorithms
\ref{twogrid-alg} and \ref{twospace-alg} to check the efficiencies
of Algorithms \ref{twogrid-alg} and \ref{twospace-alg} for  the eigenvalue problem
(\ref{problem-eq}).

\subsection{Two-gird method}\label{unitsquare}
 In the first example, we consider the problem
(\ref{problem-eq}) on the unit square $\Omega=(0,1)^2$. It is known that the eigenvalue problem
has the  following eigenpairs 
\begin{eqnarray*}
&&\lambda= (m^2+n^2)\pi^2, \ \ \ \ \ u=\sin(m\pi x)\sin(n\pi y),
\end{eqnarray*}
where $m$, $n$ are arbitrary positive integers. The first four different eigenvalues
are $\lambda_1=2\pi^2$, $\lambda_2=5\pi^2$, $\lambda_3=8\pi^2$ and
$\lambda_4=10\pi^2$, where algebraic or geometric multiplicities for
$\lambda_1$ and $\lambda_3$ are $1$ and for
$\lambda_2$ and $\lambda_4$ are both $2$. 

The uniform mesh is applied in the following examples, $H$ and $h$
denote mesh sizes. Numerical results for different choices of the parameter $\eps$ and
the degree $k$ of polynomial are presented. The corresponding
numerical results are showed in Tables \ref{Tables_Exam1_1}-\ref{Tables_Exam1_7}.
In Tables \ref{Tables_Exam1_1}-\ref{Tables_Exam1_2}, the polynomial
degree $k=1$, and $\varepsilon$ is set to be $0$ and $0.1$, separately.
{\color{black}From Theorem \ref{Error_Eigen_Two_Grid_Theorem}, we know the convergence
order for eigenvalue approximation is $2-2\varepsilon$ which is
shown from the numerical results included in Tables \ref{Tables_Exam1_1}
for $\varepsilon=0$ and \ref{Tables_Exam1_2} for $\varepsilon=0.1$. }
In Tables \ref{Tables_Exam1_3}-\ref{Tables_Exam1_7}, the
polynomial degree $k=2$ and $\varepsilon=0.1$. The mesh size $h$ is
selected to be $H^2$ in Tables \ref{Tables_Exam1_3}-\ref{Tables_Exam1_4}, and $H^{\frac32}$ in
Tables \ref{Tables_Exam1_6}-\ref{Tables_Exam1_7}. The convergence orders
for the eigenvalues, the trip-bar norm of eigenfunctions are presented.
The convergence orders in Tables \ref{Tables_Exam1_3}-\ref{Tables_Exam1_4}
coincide with that pblackicted in Theorem \ref{Error_Eigen_Two_Grid_Theorem}.
Since the choices of $k=2$, $\varepsilon=0.1$ and $h=H^2$ or $h=H^{\frac{3}{2}}$
do not satisfy the condition of Theorem \ref{Lower_Bound_Theorem_TwoGrid},
it is not surprising that the eigenvalue
approximations $\tilde\lambda_{j,h}$ ($j=1,...,6$) are not the lower bounds of the corresponding
exact eigenvalues (see Tables \ref{Tables_Exam1_3} and \ref{Tables_Exam1_6}).

{\color{black}Furthermore, the choice of $\eps$ can really affect the convergence
order which means the error estimates  in (\ref{Difference_twogrid_direct_eigenfunction}),
(\ref{Difference_twogrid_solving_eigenvalue}), (\ref{Error_Eigenfunction_Two_Grid}),
and (\ref{Error_Eigen_Two_Grid}) are reasonable.}


\begin{table}[]
\centering
\caption{The eigenvalue errors $\lambda-\tilde{\lambda}_h$ for Example
1 with $k=1, \eps=0$.}\label{Tables_Exam1_1}
\begin{tabular}{|c|c|c|c|c|c|c|c|c|c|c|c|c|c|}
\hline
$H$     &    1/4   &    1/8   &    1/16  \\ \hline
$h$     &    1/16  &    1/64  &    1/256 \\ \hline
$\lambda_1-\tilde \lambda_{1,h}$ & 2.1554e-1 & 1.3006e-2 & 8.0627e-4 \\ \hline
order &       & 4.0507  & 4.0118  \\ \hline
$\lambda_2-\tilde \lambda_{2,h}$ & 1.3687e+0 & 8.2219e-2 & 4.8684e-3 \\ \hline
order &       & 4.0572  & 4.0780  \\ \hline
$\lambda_3-\tilde \lambda_{3,h}$ & 1.3687e+0 & 7.8229e-2 & 4.8240e-3 \\ \hline
order &       & 4.1290  & 4.0194  \\ \hline
$\lambda_4-\tilde \lambda_{4,h}$ & 3.1148e+0 & 2.0798e-1 & 1.2318e-2 \\ \hline
order &       & 3.9046  & 4.0776  \\ \hline
$\lambda_5-\tilde \lambda_{5,h}$ & 4.3750e+0 & 3.0337e-1 & 1.8860e-2 \\  \hline
order &       & 3.8501  & 4.0077  \\ \hline
$\lambda_6-\tilde \lambda_{6,h}$ & 4.0896e+0 & 3.0980e-1 & 1.8206e-2 \\ \hline
order &       & 3.7225  & 4.0889  \\ \hline
\end{tabular}
\end{table}

\begin{table}[]
\centering
\caption{The errors for the eigenvalue approximation $\tilde{\lambda}_h$
for Example 1 with $k=1, \eps=0.1$.}\label{Tables_Exam1_2}
\begin{tabular}{|c|c|c|c|c|c|c|c|c|c|c|c|c|c|}
\hline
$H$     &    1/4   &    1/8   &    1/16  \\ \hline
$h$     &    1/16  &    1/64  &    1/256 \\ \hline
$\lambda_1-\tilde \lambda_{1,h}$ & 2.7369e-1 & 1.8802e-2 & 1.3194e-3 \\ \hline
order &       & 3.8636  & 3.8330  \\ \hline
$\lambda_2-\tilde \lambda_{2,h}$ & 1.7347e+0 & 1.1960e-1 & 8.0009e-3 \\ \hline
order &       & 3.8584  & 3.9019  \\ \hline
$\lambda_3-\tilde \lambda_{3,h}$ & 1.7347e+0 & 1.1339e-1 & 7.9261e-3 \\ \hline
order &       & 3.9353  & 3.8386  \\ \hline
$\lambda_4-\tilde \lambda_{4,h}$ & 3.9686e+0 & 3.0206e-1 & 2.0167e-2 \\ \hline
order &       & 3.7157  & 3.9048  \\ \hline
$\lambda_5-\tilde \lambda_{5,h}$ & 5.7315e+0 & 4.4163e-1 & 3.1474e-2 \\ \hline
order &       & 3.6980  & 3.8106  \\ \hline
$\lambda_6-\tilde \lambda_{6,h}$ & 5.1859e+0 & 4.5160e-1 & 2.9977e-2 \\ \hline
order &       & 3.5215  & 3.9131  \\ \hline
\end{tabular}
\end{table}

\begin{table}[]
\centering
\caption{The eigenvalue errors $\lambda_h-\tilde{\lambda}_h$
for Example 1 with $k=2, \eps=0.1$.}\label{Tables_Exam1_3}
\begin{tabular}{|c|c|c|c|c|c|c|c|c|c|c|c|c|c|}
\hline
$H$     &    1/4   &    1/8   &    1/16  \\ \hline
$h$     &    1/16  &    1/64  &    1/256 \\ \hline
$\lambda_1-\tilde \lambda_{1,h}$ & 3.2380e-4 & 1.8403e-6 & 8.5157e-9 \\ \hline
order &       & 7.4590  & 7.7556  \\ \hline
$\lambda_2-\tilde \lambda_{2,h}$ & -6.9598e-2 & -2.2124e-4 & -8.7753e-7 \\ \hline
order &       & 8.2973  & 7.9779  \\ \hline
$\lambda_3-\tilde \lambda_{3,h}$ & -6.5952e-2 & -1.8022e-4 & -7.0453e-7 \\ \hline
order &       & 8.5155  & 7.9989  \\ \hline
$\lambda_4-\tilde \lambda_{4,h}$ & -1.5619e+0 & -5.5200e-3 & -2.1488e-5 \\ \hline
order &       & 8.1445  & 8.0050  \\ \hline
$\lambda_5-\tilde \lambda_{5,h}$ & -2.8009e+0 & -1.0533e-2 & -3.7459e-5 \\ \hline
order &       & 8.0549  & 8.1354  \\ \hline
$\lambda_6-\tilde \lambda_{6,h}$ & -5.3281e-1 & -1.8470e-3 & -5.4252e-6 \\ \hline
order &       & 8.1723  & 8.4113  \\ \hline
\end{tabular}
\end{table}

\begin{table}[]
\centering
\caption{The eigenfunction errors for Example 1 with $k=2, \eps=0.1$.}\label{Tables_Exam1_4}
\begin{tabular}{|c|c|c|c|c|c|c|c|c|c|c|c|c|c|}
\hline
$H$     &    1/4   &    1/8   &    1/16  \\ \hline
$h $    &    1/16  &    1/64  &    1/256 \\ \hline
$\trb{Q_h u_1-\tilde u_{1,h}}$ & 4.2020e-2 & 2.5571e-3 & 1.6773e-4 \\ \hline
order &       & 4.0385  & 3.9303  \\ \hline
$\trb{Q_h u_2-\tilde u_{2,h}}$ & 3.4429e-1 & 1.9352e-2 & 1.2275e-3 \\ \hline
order &       & 4.1531  & 3.9786  \\ \hline
$\trb{Q_h u_3-\tilde u_{3,h}}$ & 3.4382e-1 & 1.9351e-2 & 1.2275e-3 \\ \hline
order &       & 4.1512  & 3.9786  \\ \hline
$\trb{Q_h u_4-\tilde u_{4,h}}$ & 2.1940e+0 & 1.3372e-1 & 8.3730e-3 \\ \hline
order &       & 4.0363  & 3.9973  \\ \hline
$\trb{Q_h u_5-\tilde u_{5,h}}$ & 2.7400e+0 & 1.6506e-1 & 9.9455e-3 \\ \hline
order &       & 4.0531  & 4.0528  \\ \hline
$\trb{Q_h u_6-\tilde u_{6,h}}$ & 2.7293e+0 & 1.6504e-1 & 9.9455e-3 \\ \hline
order &       & 4.0476  & 4.0526  \\ \hline
\end{tabular}
\end{table}

\begin{table}[]
\centering
\caption{The eigenvalue errors $\lambda_h-\tilde{\lambda}_h$
for Example 1 with $k=2, \eps=0.1$.}\label{Tables_Exam1_6}
\begin{tabular}{|c|c|c|c|c|c|c|c|c|c|c|c|c|c|}
\hline
$H $    &    1/4   &    1/16  &    1/64  \\ \hline
$h $    &    1/8   &    1/64  &    1/512 \\ \hline
$\lambda_1-\tilde \lambda_{1,h}$ & 1.3127e-2 & 4.1262e-6 & 2.5784e-9 \\ \hline
order &       & 5.8177  & 5.3221  \\ \hline
$\lambda_2-\tilde \lambda_{2,h}$ & 1.3204e-1 & 7.3767e-5 & 2.3883e-8 \\ \hline
order &       & 5.4029  & 5.7964  \\ \hline
$\lambda_3-\tilde \lambda_{3,h}$ & 8.1757e-2 & 5.4651e-5 & 1.8027e-8 \\ \hline
order &       & 5.2734  & 5.7829  \\ \hline
$\lambda_4-\tilde \lambda_{4,h}$ & -1.0040e+0 & 2.4182e-4 & 8.0721e-8 \\ \hline
order &       & 6.0098  & 5.7744  \\ \hline
$\lambda_5-\tilde \lambda_{5,h}$ & -1.7621e+0 & 4.8559e-4 & 1.5872e-7 \\ \hline
order &       & 5.9126  & 5.7895  \\ \hline
$\lambda_6-\tilde \lambda_{6,h}$ & 1.0614e-1 & 5.1734e-4 & 1.5931e-7 \\ \hline
order &       & 3.8403  & 5.8325  \\ \hline
\end{tabular}
\end{table}

\begin{table}[]
\centering
\caption{The eigenfunction errors
for Example 1 with $k=2, \eps=0.1$.}\label{Tables_Exam1_7}
\begin{tabular}{|c|c|c|c|c|c|c|c|c|c|c|c|c|c|}
\hline
$H$     &    1/4   &    1/16  &    1/64  \\ \hline
$h $    &    1/8   &    1/64  &    1/512 \\ \hline
$\trb{Q_h u_1-\tilde u_{1,h}}$ & 1.2856e-1 & 1.9489e-3 & 3.3757e-5 \\ \hline
order &       & 3.0218  & 2.9257  \\ \hline
$\trb{Q_h u_2-\tilde u_{2,h}}$ & 6.8967e-1 & 7.9977e-3 & 1.3471e-4 \\ \hline
order &       & 3.2151  & 2.9458  \\ \hline
$\trb{Q_h u_3-\tilde u_{3,h}}$ & 6.8499e-1 & 7.9974e-3 & 1.3471e-4 \\ \hline
order &       & 3.2102  & 2.9458  \\ \hline
$\trb{Q_h u_4-\tilde u_{4,h}}$ & 2.4166e+0 & 1.8197e-2 & 2.7328e-4 \\ \hline
order &       & 3.5265  & 3.0286  \\ \hline
$\trb{Q_h u_5-\tilde u_{5,h}}$ & 3.0868e+0 & 2.5302e-2 & 3.8628e-4 \\ \hline
order &       & 3.4653  & 3.0167  \\ \hline
$\trb{Q_h u_6-\tilde u_{6,h}}$ & 3.0298e+0 & 2.5299e-2 & 3.8628e-4 \\ \hline
order &       & 3.4520  & 3.0167  \\ \hline
\end{tabular}
\end{table}

\subsection{Two-space method}

In the second example, the analytic solution is the same as
(\ref{unitsquare}). The polynomials of degree $k_1=1,~k_2=2$ and $k_1=2,~k_2=3$
are employed in $V_h^1$ and $V_h^2$, respectively. The parameter $\eps$
is chosen to be $0.2$. The results are listed in Figures
\ref{fig:1}-\ref{fig:2} for the case $k_1=1$, $k_2=2$ and $\eps=0.2$
and Tables \ref{fig:3}-\ref{fig:4} for $k_1=2$, $k_2=3$ and $\eps=0.2$.

\begin{figure}[ht]
\centering\includegraphics[width=8cm]{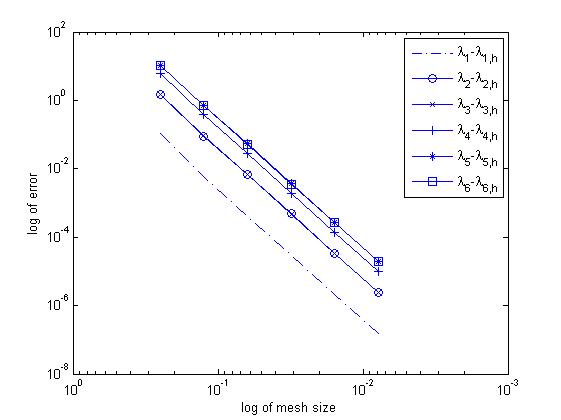}
\caption{The eigenvalue errors $\lambda_h-\hat{\lambda}_h$ for
Example 2 with $k_1=1$, $k_2=2$ and $\eps=0.2$.}\label{fig:1}
\end{figure}

\begin{figure}[ht]
\centering\includegraphics[width=8cm]{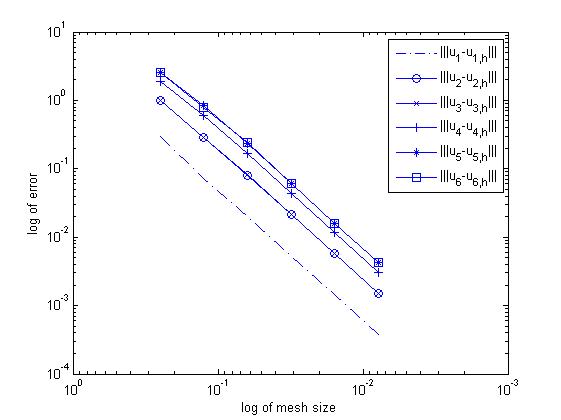}
\caption{The errors for the eigenfunction approximations $\trb{u_h-\hat u_h}$
for Example 2 with $k_1=1$, $k_2=2$ and $\eps=0.2$.}\label{fig:2}
\end{figure}

\begin{figure}[ht]
\centering\includegraphics[width=8cm]{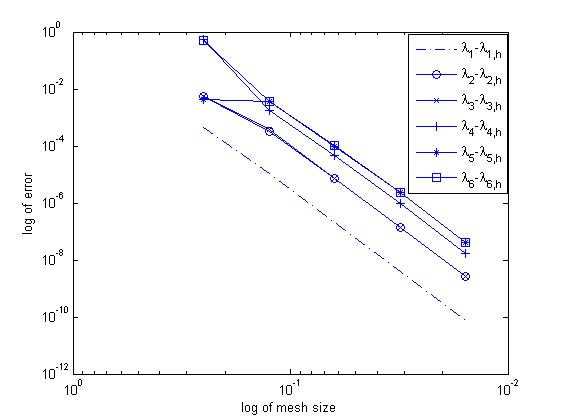}
\caption{The eigenvalue errors $\lambda_h-\hat{\lambda}_h$ for
Example 2 with $k_1=2$, $k_2=3$ and $\eps=0.2$.}\label{fig:3}
\end{figure}

\begin{figure}[ht]
\centering\includegraphics[width=8cm]{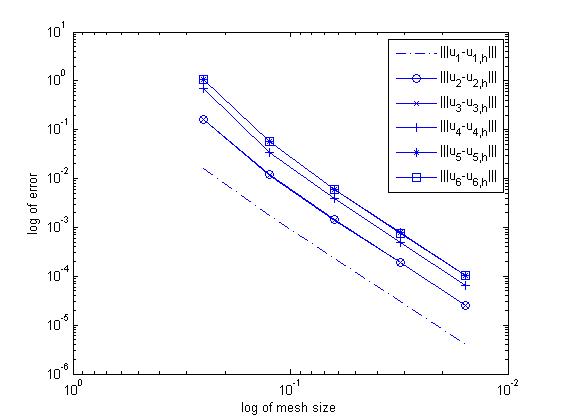}
\caption{The errors for the eigenfunction approximations $\trb{u_h-\hat u_h}$
for Example 2 with $k_1=2$, $k_2=3$ and $\eps=0.2$.}\label{fig:4}
\end{figure}

{\color{black} The convergence orders shown in Figures \ref{fig:1}-\ref{fig:4}
are consistent with the results in Theorems \ref{Two_Space_Thm} and \ref{Lower_Bound_Theorem_TwoSpace}.
Even the choices of $k_1=1$, $k_2=2$ and $\varepsilon=0.2$ do not satisfy
the condition of Theorem \ref{Lower_Bound_Theorem_TwoSpace}, the eigenvalue approximations
by the two-space method are still the lower bounds of the exact eigenvalues.}



\subsection{L-shape}
In the third example, we consider the problem
(\ref{problem-eq}) on the L-shape domain $\Omega=(-1,1)^2\slash [0,1)^2$. Since the exact eigenvalues
are unknown. We only check the eigenvalues $\tilde\lambda_{j,h}$ {\color{black}($j=1,...,6$)}.
The corresponding numerical results are shown in Table
\ref{Tables_Exam3_1}. From Table \ref{Tables_Exam3_1}, we find
that the two-grid method defined in Algorithm \ref{twogrid-alg}
is accurate and provides lower bounds.

\begin{table}[]
\centering
\caption{The errors for the eigenvalue approximations $\lambda_h-\tilde\lambda_h$
for Example 3 with $k=2, \eps=0.1$.}\label{Tables_Exam3_1}
\begin{tabular}{|c|c|c|c|c|c|c|c|c|c|c|c|c|c|}
\hline
$H$     &    1/4   &    1/16  &    1/64  & Trend\\ \hline
$h$     &    1/8   &    1/64  &    1/512 & \\ \hline
$\tilde \lambda_{1,h}$ & 9.6152615304  & 9.6383056544  & 9.6396344695  & $\nearrow$\\ \hline
$\tilde \lambda_{2,h}$ & 15.1905227597  & 15.1972465939  & 15.1972519114  & $\nearrow$\\ \hline
$\tilde \lambda_{3,h}$ & 19.7262367431  & 19.7392046788  & 19.7392088004  & $\nearrow$\\ \hline
$\tilde \lambda_{4,h}$ & 29.4848098848  & 29.5214666179  & 29.5214811041  & $\nearrow$\\ \hline
$\tilde \lambda_{5,h}$ & 31.7992824737  & 31.9091062924  & 31.9124163062  & $\nearrow$\\ \hline
$\tilde \lambda_{6,h}$ & 41.3482606285  & 41.4717757164  & 41.4743429661  & $\nearrow$\\ \hline
\end{tabular}
\end{table}

\section{Concluding remarks and ongoing work}
In this paper, we propose and analyze the two-grid and two-space schemes for the
eigenvalue problem by the WG method. Based on our analysis, the
eigenpair approximations by the two-grid and two-space methods possess the same
reasonable accuracy as the direct WG approximations, but the calculation
cost is significantly reduced. From the
numerical examples, we also find that the eigenvalue
approximations by the two-grid method have the same lower bound
property as the direct WG approximations, if we choose the grid or
space properly.

In the future work, we are going to study the shift-inverse power
method and multigrid method for the Laplacian eigenvalue problem,
and other kinds of eigenvalue problems, such as biharmonic eigenvalue problems and Stokes
eigenvalue problems.

\bibliographystyle{siam}
\bibliography{library}

\end{document}